\documentclass[10pt]{article}

\usepackage{amsmath,amsthm,amssymb,amsbsy,amsfonts,graphics,epsfig,graphicx,mathabx}
\usepackage{color}

\usepackage{etoolbox}
\let\bbordermatrix\bordermatrix
\patchcmd{\bbordermatrix}{8.75}{4.75}{}{}
\patchcmd{\bbordermatrix}{\left(}{\left[}{}{}
\patchcmd{\bbordermatrix}{\right)}{\right]}{}{}

\newcommand{\R}{{\mathbb R}}

\newcommand{\Oo}{{\cal O}}

\newcommand{\beq}{\begin{equation}}
\newcommand{\eeq}{\end{equation}}
\newcommand{\bea}{\begin{eqnarray*}}
\newcommand{\eea}{\end{eqnarray*}}
\newcommand{\beqa}{\begin{eqnarray}}
\newcommand{\eeqa}{\end{eqnarray}}
\newcommand{\bmat}{\begin{bmatrix}}
\newcommand{\emat}{\end{bmatrix}}
\newcommand{\ba}{\begin{align}}
\newcommand{\ea}{\end{align}}
\newcommand{\bas}{\begin{align*}}
\newcommand{\eas}{\end{align*}}

\newcommand{\dsp}{\displaystyle}

\newcommand{\ve}{\varepsilon}

\newtheorem{theo}{{\bf Theorem}}
\newtheorem*{theo*}{{\bf Theorem}}
\newtheorem{lem}{{\bf Lemma}}
\newtheorem{prob}{{\bf Problem}}

\theoremstyle{definition}
\newtheorem{defn}[theo]{{\bf Definition}}
\newtheorem{rem}{{\bf Remark}}

\newcommand{\bth}{\begin{theo}}
\renewcommand{\eth}{\end{theo}}
\newcommand{\bdefi}{\begin{defn}}
\newcommand{\edefi}{\end{defn}}
\newcommand{\bprob}{\begin{prob}}
\newcommand{\eprob}{\end{prob}}

\author{Alain Bourget\thanks{\textbf{Mailing address}: Department of Mathematics, California State University (Fullerton),
McCarthy Hall 154, Fullerton CA 92834 (US). \textbf{Email address}: abourget@fullerton.edu,  tmcmillen@fullerton.edu } \ and Tyler McMillen\footnotemark[1] }

\title{Asymptotics of determinants of discrete Schr\"{o}dinger operators}

\begin{document}

\maketitle

\bibliographystyle{plain}
 
\begin{abstract}

We consider the asymptotics of the determinants of large discrete Schr\"{o}\-din\-ger operators, i.e. ``discrete Laplacian $+$ diagonal'':  
\[T_n(f) = -[\delta_{j,j+1}+\delta_{j+1,j}] + \mbox{diag}(f(1/n), f(2/n),\dots, f(n/n))  
\]
We extend a result of M. Kac \cite{ka69} who found a formula for 
\[\dsp \lim_{n\rightarrow\infty} \frac{\det(T_n(f))}{G(f)^n}
\]
in terms of the values of $f$, where $G(f)$ is a constant.  We extend this result in two ways:  First, we consider shifting the index:  Let 
\[T_n(f;\varepsilon) =  -[\delta_{j,j+1}+\delta_{j+1,j}]  + \mbox{diag}\left(f\left(\frac{\varepsilon}{n}\right), f\left(\frac{1+ \varepsilon}{n}\right), \dots , f\left(\frac{n-1+ \varepsilon}{n}\right)\right)  
\]
We calculate $\lim \det T_n(f;\ve)/G(f)^n$ and show that this limit can be any positive number by shifting $\varepsilon$, even though the asymptotic eigenvalue distribution of $T_n(f;\varepsilon)$ does not depend on $\varepsilon$.  Secondly, we  derive a formula for the asymptotics of $\det T_n(f)/G(f)^n$ when $f$ has jump discontinuities.  In this case the asymptotics depend on the fractional part of $c n$, where $c$ is a point of discontinuity.
\end{abstract}

\bigskip

\bigskip


%
%
\medskip
\section {Introduction and main results }
\label{DSOsection}
%
%


This paper is concerned with a remarkable and little known result of M. Kac on the asymptotics of the determinant of the discrete Schr\"{o}dinger operator
\beq 
T_n(f) = \begin{bmatrix} f(\frac{1}{n}) & -1 & 0 & \cdots & 0 \\
-1 & f(\frac{2}{n}) & -1 & \cdots & 0 \\
0 & -1 & f(\frac{3}{n}) & \cdots & 0 \\
 &  &  & \ddots
 \\
0 & 0 & 0 & \cdots & f(\frac{n}{n}) 
\end{bmatrix}
\label{DSOdef}
\eeq

By a result of Kac, Murdock and Szeg\H{o} \cite{kamusz53}, the following holds for the trace.  As long as $f$ is real valued and Riemann integrable, we have 
\begin{equation} \label{First KMS}
\lim_{n\rightarrow\infty} \frac{ \text{Tr}[\varphi(T_n(f))]}{n} = \frac{1}{2\pi} \int_0^1 \int_0^{2\pi} \varphi(f(x)-2\cos t) \, dt \, dx 
\end{equation}  
for any continuous $\varphi(s)$.  
This result says, roughly, that as $n\rightarrow\infty$, the eigenvalues of $T_n(a)$ distribute like the values of $f(x)-2\cos t$ sampled at regularly spaced points in the rectangle $0\leq x\leq 1, \ 0\leq t\leq 2\pi$.
 
The formula \eqref{First KMS} gives us some information about the determinant.  Let
\[
D_n(f) = \det T_n(f)
\]
Then, with $\varphi = \log$, as long as $f>2$, \eqref{First KMS} can be written
 \begin{equation*}
 \lim_{n\rightarrow\infty} D_n(f)^{1/n} =  G(f)
 \end{equation*}
 where 
 \[G(f) = \exp \left\{ \int_0^1 \log\left(\frac{f(x) + \sqrt{f^2(x)-4}}{2}\right) dx \right\}
\]
is the geometric mean of $f(x)-2\cos t$.
  
 \smallskip

In the early 1960's, Mejlbo and Schmidt \cite{mesc62} considered determinants of a broader class of matrices, of which \eqref{DSOdef} is a special case.  Their result implies that, as long as $f>2$ and $f\in C^{2+\alpha}([0,1])$ for some $\alpha>0$, then we have the more precise statement
\[ \lim_{n \to \infty}  \frac{D_n(f)}{G(f)^n}   = E(f)
\]
where $E(f)$ is a constant defined in the following way.  Let 
\[ V_k(f;x) = \frac{1}{2\pi} \int_0^{2\pi} \log (f(x)-2\cos t) e^{-ikt} dt
\]
be the $k$th Fourier coefficient of $\log(f(x)-2\cos t)$.  Then
\begin{align*}
E(f) &= \exp\frac{1}{2}\bigg\{ V_0(f;0)+V_0(f;1) 
\\
& \quad + \left.\sum_{k=1}^{\infty} kV_{k}(f;0)V_{-k}(f;0) + \sum_{k=1}^{\infty} kV_{k}(f;1)V_{-k}(f;1)\right\}
\end{align*}
Remarkably, $E(f)$ depends on the value of $f$ only at $x=0$ and $x=1$.

In 1969 Kac \cite{ka69} derived a beautiful and simple formula for $E(f)$ for this case.  
\begin{theo}[Kac, 1969]
\label{kacformula}
Let $f$ be twice differentiable on $[0,1]$, with a bounded second derivative, and satsify $f > 2$.  Then, 
\beq
 \lim_{n \to \infty}  \frac{D_n(f)}{G(f)^n}   = \frac{1}{2} \, \frac{f(1)+\sqrt{f^2(1) - 4}}{\sqrt[4]{(f^2(0)-4)(f^2(1)-4)}}
 \label{kac1}
\eeq
\end{theo}
\noindent
In \S\ref{sec.proofs} we will repeat Kac's proof of this theorem,  with a few details that Kac omitted.  We will then show how his proof can be modified for the two theorems below.

\begin{rem}
Kac's paper \cite{ka69} contains a typo in the formula for $\dsp\lim_{n\rightarrow\infty}\frac{D_n(f)}{G(f)^n}$.  His formula (eqn (3.15) in \cite{ka69}) is missing the factor $1/2$.
\end{rem}

\begin{rem}
Kac's result can be viewed as a Szeg\"o Strong Limit Theorem (SSLT) for the matrices in \eqref{DSOdef}. In the past few decades, the SSLT has been extensively used to study the spectral theory of discrete Schr\"odinger operators.  See, for example, the recent book of Simon \cite{si11} and the references therein.
\end{rem}

\begin{rem}
Theorem \ref{kacformula} holds under the slightly weaker condition that $f\in C^{1+\alpha}([0,1])$  for some $\alpha>0 $.  Kac's proof can easily be modified in this case, but it  involves some tedious technicalities, which we omit. However, see Remark~\ref{1palpha} after the proof of Theorem~\ref{kacformula} for the outline of how to modify the proof for this case.
\end{rem}

Our first extension of Theorem~\ref{kacformula} has to do with shifting the indexing.  Kac \cite{ka69} noted that if one shifts the indexing by 1, one obtains a different formula for $\lim D_n(f)/G(f)^n$.  We extend this to any shift.

\begin{theo}
\label{scaletheorem}
Let $f$ be twice differentiable on some open interval $I$ containing $[0,1]$.  Suppose $f$ has  a bounded second derivative and satisfies $f>2$.  Let $\ve\in \mathbb{R}$ and define 
the matrices
\[
T_{n}(f;\ve ) =  \begin{bmatrix} f(\frac{\ve }{n}) & -1 & 0 & \cdots & 0 \\
-1 & f(\frac{1+\ve }{n}) & -1 & \cdots & 0 \\
0 & -1 & f(\frac{2+\ve }{n}) & \cdots & 0 \\
 &  &  & \ddots
 \\
0 & 0 & 0 & \cdots & f(\frac{n-1+\ve }{n}) 
\end{bmatrix}
\]
Then 
\beq 
 \lim_{n \to \infty}  \frac{\det {T}_{n}(f;\ve)}{G(f)^n}   =\frac{\dsp \left( f(0)+\sqrt{f^2(0 ) - 4}\right)^{1-\ve} \left( f(1)+\sqrt{f^2(1 ) - 4}\right)^{\ve} }{ 2 \dsp
 {\sqrt[4]{(f^2(0)-4)(f^2(1)-4)}}}
\label{dsoe}
\eeq
\end{theo}

\medskip
\begin{rem}
When $\ve =1$, \eqref{dsoe} reduces to \eqref{kac1}.
As long as $f(0)\neq f(1)$, the above limit can be made to converge to any positive number just by choosing the correct shift $\ve$.  Notice that the limiting statistical distribution of the eigenvalues of $T_n(f;\ve)$ does not depend on $\ve$.  The result \eqref{First KMS}   holds for $T_n(f;\ve)$ for any $\ve$:
\[ \lim_{n\rightarrow\infty} \frac{ \text{Tr}[\varphi(T_n(f;\ve ))]}{n} = \frac{1}{2\pi} \int_0^1 \int_0^{2\pi} \varphi(f(x)-2\cos t) \, dt \, dx 
\]
This is therefore an example of a family of matrices whose asymptotic eigenvalue distribution is invariant, but the asymptotics of the determinants can be made to converge to any positive number.  
\label{shiftremark}
\end{rem}

\begin{rem}
The formulas \eqref{kac1} and \eqref{dsoe} remain unchanged if we take the super-and sub-diagonals to be $+1$, instead of $-1$. 
\end{rem}

Given how sensitive the limits \eqref{kac1} and \eqref{dsoe} are to the slightest change, it might seem that we have no hope of deriving a limit for the determinant  when the symbol is discontinuous.  However, there is one important case when we can do it.
When the function $f$ has a finite number of jump discontinuities, we can obtain a formula for $D_n(f)/G(f)^n$ that depends on $n$ (modulo $o(1)$ terms).  This demonstrates the impossibility of the limit of $D_n(f)/G(f)^n$ existing when $f$ is discontinuous.

\begin{theo}
(i) Let $f$ be twice differentiable on $[0,1]$, with a bounded second derivative,  except for  $r<\infty$ jump discontinuities at $c_1,\dots, c_r \in (0,1)$, where both sided limits exist and are finite, and $f$ is left-continuous at $c_j$: $f(c_j)=f(c_j-)$.  Suppose, also, that  $f>2+\epsilon$ for some $\epsilon > 0$.  Then

\beq
  \frac{D_n(f)}{G(f)^n} = \alpha \prod_{j=1}^r \beta_j \gamma_j^{\{nc_j\}} + o(1)
  \label{jumpeqn}
\eeq
where $\{x\}=x-\lfloor x \rfloor$ is the fractional part of $x$,
\[\alpha = \frac{1}{2} \frac{f(1)+\sqrt{f^2(1) - 4}}{\sqrt[4]{(f^2(0)-4)(f^2(1)-4)}} 
\]
as in \eqref{kac1}, 
\begin{align*}
\beta_j &= \frac{f(c_j-)-f(c_j+) + \sqrt{f^2(c_j+)-4} + \sqrt{f^2(c_j-)-4}}{2\sqrt[4]{(f^2(c_j+)-4)(f^2(c_j-)-4)}}
 \\[.1in]
\mbox{and } \quad \gamma_j &= \frac{f(c_j+)+\sqrt{f^2(c_j+)-4}}{f(c_j-)+\sqrt{f^2(c_j-)-4}}
\end{align*}

\smallskip
\noindent
(ii) If $f$ is right-continuous at $c_j$, then the formula \eqref{jumpeqn} holds with $\{c_jn\}$ replaced by $\{c_jn\}'$, where 
\[ \{x\}' =
1+x-\lceil x\rceil
\]
is the fractional part of $x$, but equal to $1$ if $x$ is an integer.
\label{jumptheorem}
\end{theo}

\begin{rem}
Note that $\beta_j$ and $\gamma_j$ are $1$ if $f$ is continuous at $c_j$, so \eqref{jumpeqn} reduces to \eqref{kac1} when $f$ is smooth.  Since $\{ c_jn\} =\{c_jn\}'$ if $c_j$ is irrational,  the difference between cases (i) and (ii) of the above theorem only occurs when $c_j$ is rational.  In that case the difference arises when $f$ is evaluated at the point $c_j$.  
\end{rem}

\begin{rem}
Obviously, if there is a discontinuity in $f$, the limit
\[ \lim_{n\rightarrow\infty} \frac{D_n(f)}{G(f)^n}
\]
does not exist.  However, we can calculate the $\limsup$ and $\liminf$.  For example, if there is one jump discontinuity at  $c=p/q$,
\begin{align*}
\limsup_{n\rightarrow\infty} \frac{D_n(f)}{G(f)^n} &= \alpha\cdot \beta\cdot \max\{ \gamma^{1/q}, \gamma\}
\\
\liminf_{n\rightarrow\infty} \frac{D_n(f)}{G(f)^n} &= \alpha\cdot \beta \cdot \min\{ \gamma^{1/q}, \gamma\}
\end{align*}
If $c$ is irrational, the same is true with $\gamma^{1/q}$ replaced by $1$.  Analogous statements hold when there are $r$ jump discontinuities.
\end{rem}

To illustrate the asymptotic behavior of $D_n(f)/G(f)^n$, we
consider the case of a single jump discontinuity at $c\in (0,1)$.
If $c =p/q$ is rational,  $\{D_n(f)/G(f)^n\}$ (modulo an $o(1)$ term) is cyclic of order $q$.    When $c$ is irrational, $\{D_n(f)/G(f)^n\}$ is dense on the interval  between $\alpha\beta$ and $\alpha\beta\gamma$.  This is another indication of how exquisitely sensitive  $D_n(f)/G(f)^n$ is.  The slightest irrational perturbation of the point of discontinuity from $c=1/2$, causes the values of $D_n(f)/G(f)^n$ (modulo the $o(1)$ term) to go from alternating between two values to taking on infinitely many values.  This behavior is illustrated in figure~\ref{jumpfig}.  There we calculate $D_n(f)/G(f)^n$ for the piecewise function
\beq 
f(x) = \begin{cases} 3.3+x^2/2 + \sin(3 x) & x \leq c \\
3.5-x & x > c \end{cases}
\label{ff}
\eeq
We compare the values of $D_n(f)/G(f)^n$ with $\alpha\beta\gamma^{\{cn\}'}$.  Agreement is quite good for moderately large $n$.

\begin{figure}[h]
\begin{center}
%
\includegraphics*[width=1.6in]{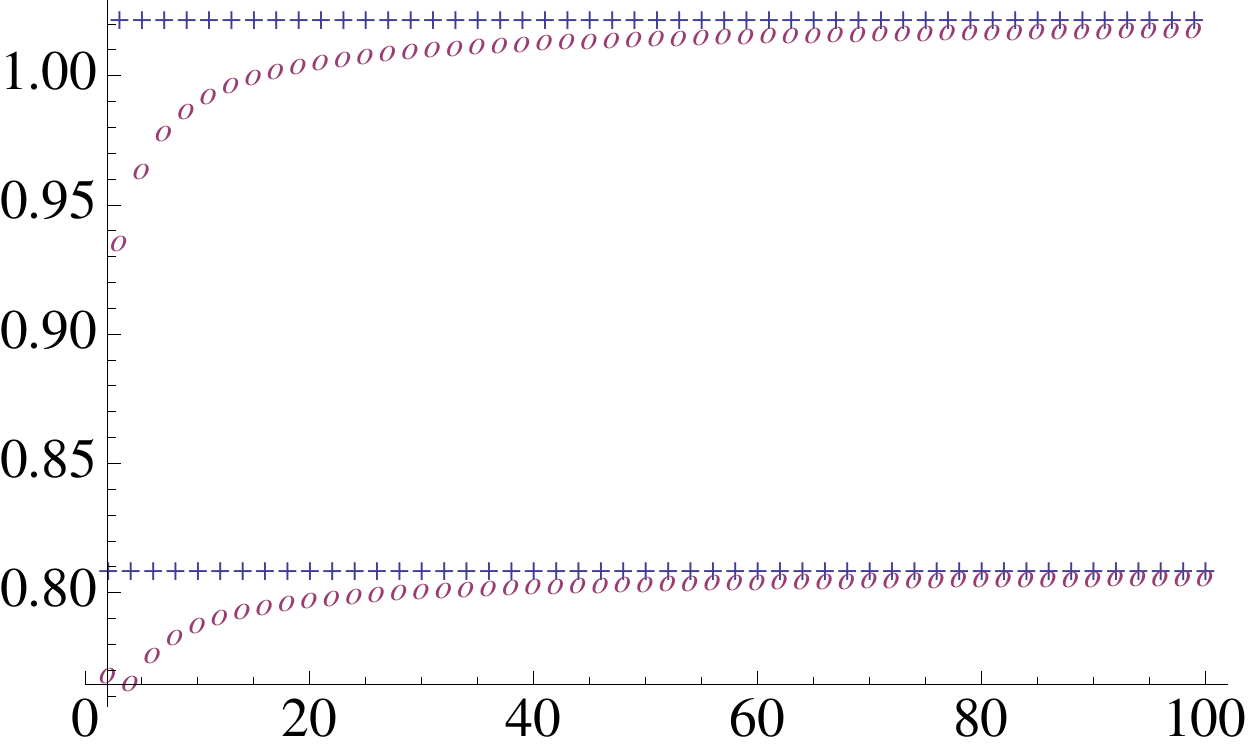}
\includegraphics*[width=1.6in]{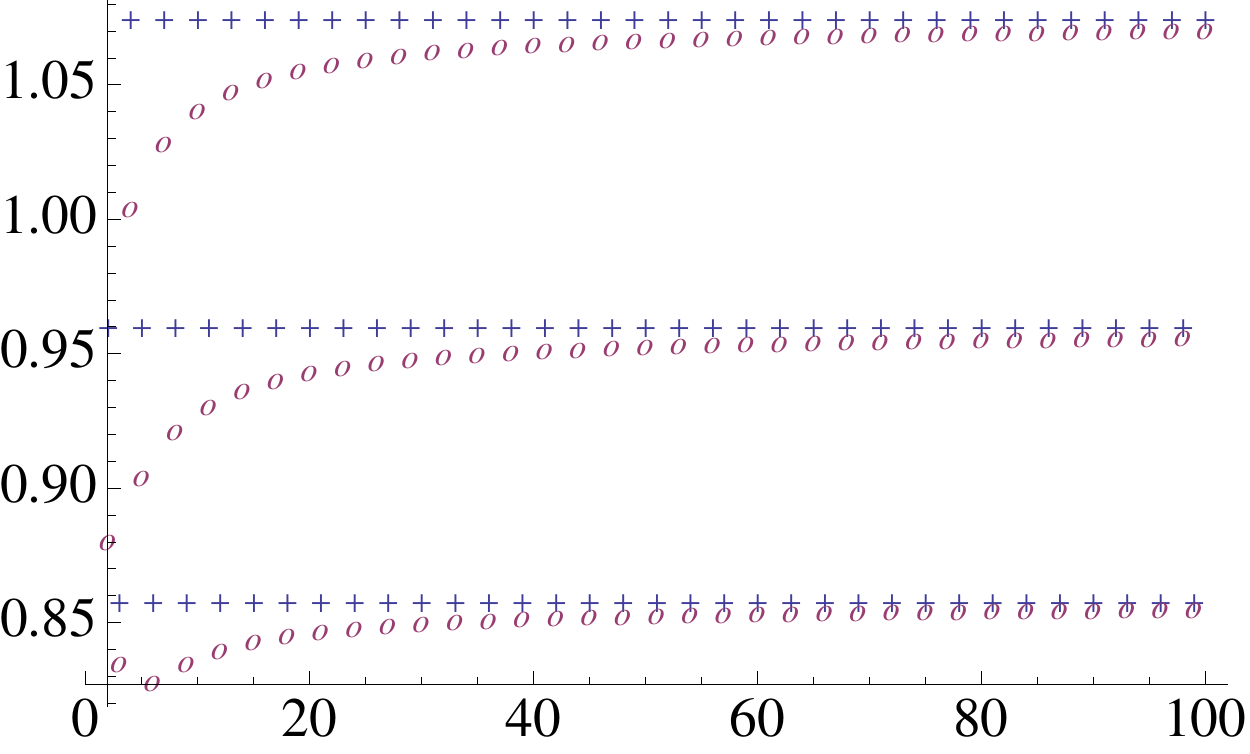}
\includegraphics*[width=1.6in]{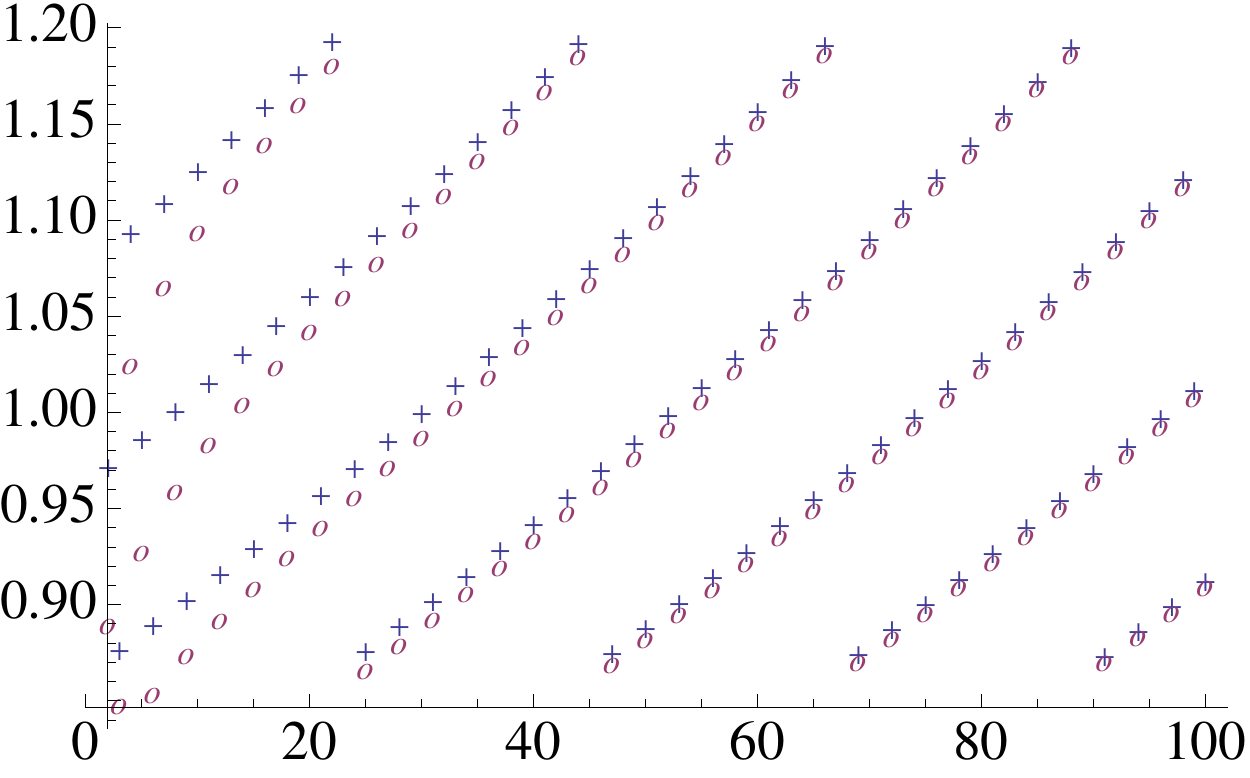}
\end{center}
\caption{$n$ vs. $D_n(f)/G(f)^n$  for $f$ as in \eqref{ff}.  Left: $c=1/2$; middle: $c=1/3$; right: $c=1/\pi$.   The values of $D_n(f)/G(f)^n$ are marked with circles;  the values of $\alpha\beta\gamma^{\{cn\}'}$ are marked with $+$'s.  Note that the values of $c$ in the middle and right panels differ by less than $.0151$.}
\label{jumpfig}
\end{figure}

As another example, to illustrate the behavior of the $o(1)$ error in \eqref{jumpeqn}, we take the function
\beq
f(x) = \begin{cases} 3.3+x^2/2 + \sqrt{x}\sin(13 x) & x \leq .9-1/\pi \\
3.5 - \cos(20 x) & x > .9-1/\pi
\end{cases}
\label{ff2}
\eeq
In figure~\ref{jumpfig2} we plot the
values of $D_n(f)/G(f)^n$ and $\alpha\beta\gamma^{\{cn\}}$ for $n$ to 200 (left panel) and the difference between these values for $n$ to 3000 (right panel).  Among power laws and exponential functions of the form $A B^n$ and $An^{b}$, we found the best least square fit to the data $\left\{ \left(n,D_n(f)/G(f)^n - \alpha\beta\gamma^{\{cn\}} \right)\right\}$ in figure~\ref{jumpfig2} to be  $2.82506\, n^{-0.965199}$.  In other words, the error approaches zero like $1/n$.

\begin{figure}[h]
\begin{center}
%
\includegraphics*[width=2.5in]{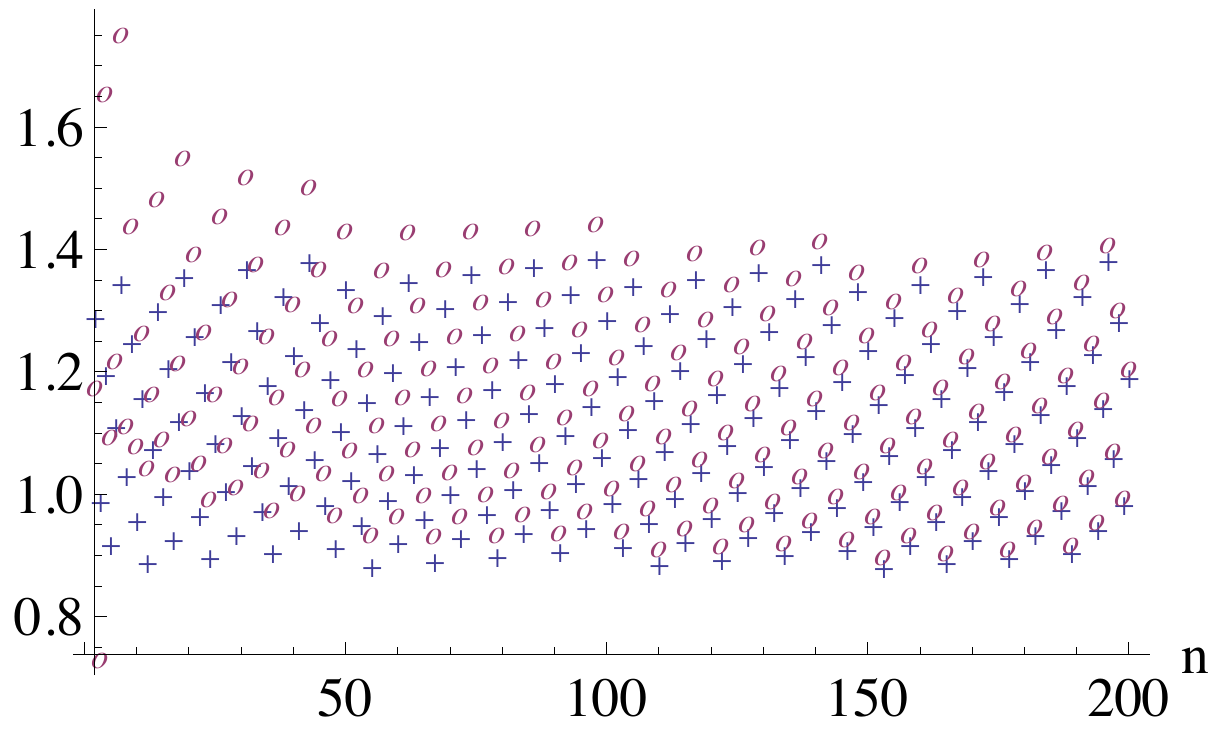}
\includegraphics*[width=2.3in]{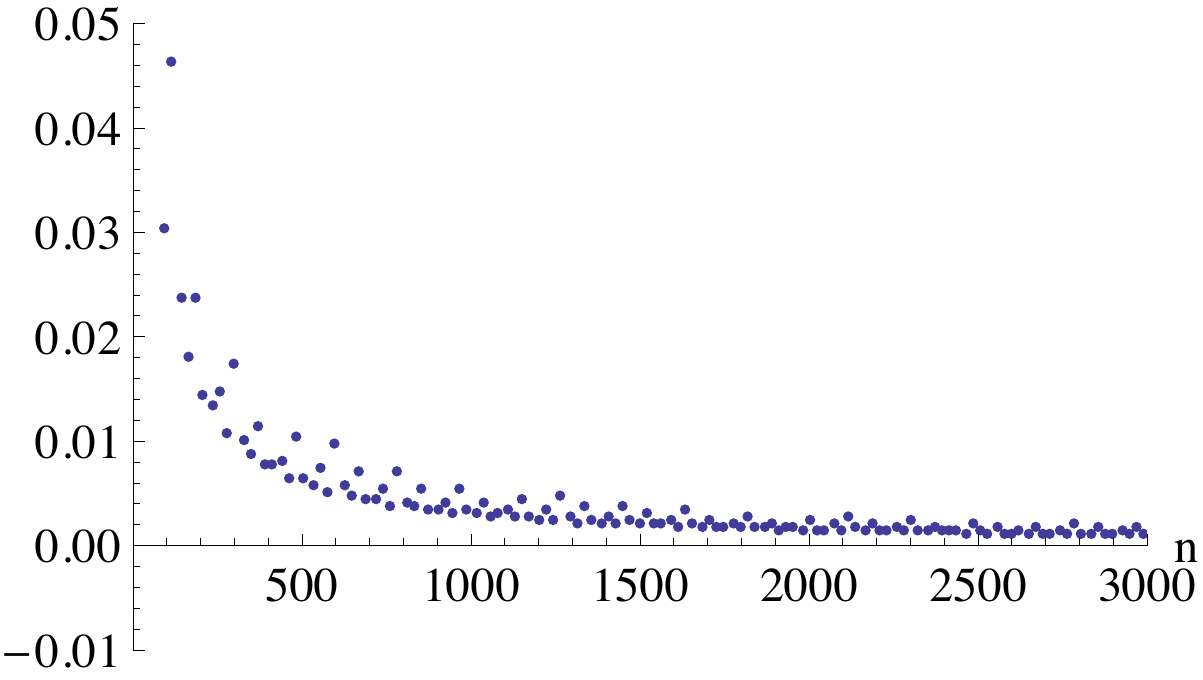}
\end{center}
\caption{Left:  $D_n(f)/G(f)^n$ (circles) and $\alpha\beta\gamma^{\{cn\}}$ ($+$'s)  for $f$ as in \eqref{ff2}.  Right:  $D_n(f)/G(f)^n - \alpha\beta\gamma^{\{cn\}}$ in steps of 23.}
\label{jumpfig2}
\end{figure}

\section{Proofs of main results}
\label{sec.proofs}
%
%

\begin{proof}[Proof of Theorem \ref{kacformula}]
Kac's proof begins with the formula for the determinant of a positive definite matrix $A$. Let $\lambda_1,\ldots,\lambda_n$ be the eigenvalues of $A$. From 
$$ \int_{-\infty}^\infty e^{-\lambda_k t^2} \, dt = \frac{\sqrt{\pi}}{\sqrt{ \lambda_k}},$$
and the spectral theorem, we see    
\begin{align} 
\label{kac2a}
\frac{1}{\sqrt{\det A}} & =   \frac{1}{\left(\sqrt{\pi}\right)^n} \int_{-\infty}^{\infty} \cdots \int_{-\infty}^{\infty} e^{-(\lambda_1y_1^2+\cdots+\lambda_ny_n^2)}  \, dy_1dy_2\cdots dy_n  \nonumber \\
  &=  \frac{1}{\left(\sqrt{\pi}\right)^n} \int_{-\infty}^{\infty} \cdots \int_{-\infty}^{\infty} \exp\left[- x^TAx \right] \, dx_1dx_2\cdots dx_n.
\end{align}
Moreover, the asymptotic expansion of the complementary error function implies that
$$ \int_{n^{1/4}}^{\infty} e^{-\lambda_k t^2} \, dt  = \Oo \left(  \frac{e^{-n^{1/2}}}{n^{1/4}} \right),$$
and hence
\begin{equation} \label{kac3}
  \frac{1}{\sqrt{\det A}} = \int_{-n^{1/4}}^{n^{1/4}} \cdots \int_{-n^{1/4}}^{n^{1/4}}  \exp \left[- x^TAx \right] \, dx_1dx_2\cdots dx_n + o(1).
\end{equation}  
This estimate will play a role below when we approximate the above integral. 
If we apply \eqref{kac2a} to $D_n(f)$, we obtain 
\begin{align}
\frac{1}{\sqrt{D_n(f)}} & = \frac{1}{\left(\sqrt{\pi}\right)^n} \int_{-\infty}^{\infty} \cdots \int_{-\infty}^{\infty}
\nonumber \\
& \quad
\exp\left[-\sum_{k=1}^n f\left(\frac{k}{n}\right) x_k^2  + 2 \sum_{k=1}^{n-1}x_kx_{k+1}\right] dx_1\cdots dx_n
\label{dfintegral}
\end{align}
To obtain a limit  we need to write the integrand as a product of symmetric kernels.   To this end, we note
\begin{multline}
f\left(\frac{k}{n}\right) =  \frac{1}{2}f\left(\frac{k-1}{n}\right) + \frac{1}{4n}f'\left(\frac{k-1}{n}\right) \\
 + \frac{1}{2}f\left(\frac{k}{n}\right)
+ \frac{1}{4n}f'\left(\frac{k}{n}\right)
+ \Oo\left(\frac{1}{n^2}\right)
\label{fexp}
\end{multline}
Let
\[ a_k = f\left(\frac{k}{n}\right) + \frac{1}{2n}f'\left(\frac{k}{n}\right)
\]
Note that $a_k > 2$ for $n$ large enough.
Then, by \eqref{fexp},  we have
\begin{align}
\sum_{k=1}^n f\left(\frac{k}{n}\right) x_k^2  &= \frac{1}{2}\left(a_0 x_0^2+a_n x_n^2\right)  \nonumber \\
 & + \frac{1}{2} \sum_{k=1}^{n-1} a_k \left(x_k^2+x_{k+1}^2\right) + \Oo\left(\sum_{k=1}^n \frac{x_k^2}{n^{2}} \right)
\label{fsum}
\end{align}
Thus, if we define the symmetric kernels
\begin{equation*}
K\left(x,y;\frac{k}{n}\right) = \frac{1}{\sqrt{\pi}} \exp\left[-\frac{a_k}{2} x^2 + 2xy - \frac{a_k}{2} y^2\right]
\end{equation*}
then we have
\begin{align*}
\frac{\sqrt{\pi}}{\sqrt{D_n(f)}} &=   \int_{-\infty}^{\infty} \cdots \int_{-\infty}^{\infty}
\\
& \quad
e^{-\frac{a_1}{2}x_1^2} \prod_{k=1}^{n-1}K\left(x_k, x_{k+1};\frac{k}{n}\right) e^{-\frac{a_n}{2} x_n^2}\, e^{\Oo \left(\frac{ \sum_{k=1}^n x_k^2}{n^{2}} \right)}
dx_1\cdots dx_n .
\end{align*}
The $\Oo$ term plays essentially no role for large $n$ and can therefore be removed. Indeed, it follows from \eqref{kac3}
\begin{align}
  & \frac{\sqrt{\pi}}{\sqrt{D_n(f)}} \nonumber \\
  & =   \int_{-\infty}^{\infty} \cdots \int_{-\infty}^{\infty}
e^{-\frac{a_1}{2}x_1^2} \prod_{k=1}^{n-1}K\left(x_k, x_{k+1};\frac{k}{n}\right) e^{-\frac{a_n}{2} x_n^2}\, e^{\Oo \left(\frac{ \sum_{k=1}^n x_k^2}{n^{2}} \right)}
dx_1\cdots dx_n  
\nonumber 
\\
& =  \int_{-n^{1/4}}^{n^{1/4}} \cdots \int_{-n^{1/4}}^{n^{1/4}}
\nonumber \\
& \qquad
e^{-\frac{a_1}{2}x_1^2} \prod_{k=1}^{n-1}K\left(x_k, x_{k+1};\frac{k}{n}\right) e^{-\frac{a_n}{2} x_n^2}\, e^{\Oo \left(\frac{ \sum_{k=1}^n x_k^2}{n^{2}} \right)}
dx_1\cdots dx_n  + o(1) 
\nonumber
\\
& =   \int_{-n^{1/4}}^{n^{1/4}} \cdots \int_{-n^{1/4}}^{n^{1/4}}
\nonumber \\
& \qquad
e^{-\frac{a_1}{2}x_1^2} \prod_{k=1}^{n-1}K\left(x_k, x_{k+1};\frac{k}{n}\right) e^{-\frac{a_n}{2} x_n^2}\, e^{\Oo \left( n^{-1/2} \right)}
dx_1\cdots dx_n  + o(1) 
\nonumber
\\
& =  \int_{-\infty}^{\infty} \cdots \int_{-\infty}^{\infty} e^{-\frac{a_1}{2}x_1^2} \prod_{k=1}^{n-1}K\left(x_k, x_{k+1};\frac{k}{n}\right) e^{-\frac{a_n}{2} x_n^2}\,  dx_1\cdots dx_n  + o(1).
\label{detTna}
\end{align}

Using the self-reciprocity of the Hermite polynomials, one can easily compute the eigenvalues and eigenfunctions of $\frac{1}{\sqrt{\pi}} \, K(x,y;k/n)$. The eigenvalues are 
\begin{equation*} \label{eigenvalue}
 \lambda_j(k/n) = \left( \frac{a_k+b_k}{2} \right)^{-j-1/2} \qquad (j=0,1,2,...)
 \end{equation*}
with $b_k=\sqrt{a_k^2-4}$, and the corresponding normalized eigenfunctions 
\begin{equation*}
  \phi_j(x;k/n) = \sqrt[4]{\frac{b_k}{\pi}} \, \frac{1}{ 2^{j/2} \sqrt{j!}} \, e^{- \frac{b_k}{2} \, x^2} \, H_j(\sqrt{b_k} \, x) \qquad (j=0,1,2,...)
\end{equation*} 
where $H_j$  is the $j$th Hermite polynomial \cite[Cf. Rem. 6.1.1]{anasro99}. Since $b_{k+1}-b_k=\Oo(n^{-2})$,  the normalized eigenfunctions ``almost commute'' in the following sense:
\begin{equation} \label{comm}
\int_{-\infty}^{\infty} \phi_i\left(x;\frac{k}{n}\right) \phi_j\left(x;\frac{k+1}{n}\right) dx = \delta_{ij} + \Oo\left(\frac{1}{n^{2}}\right)
\end{equation}
For any given $k$ and $n$, the collection $\{\phi_0(x;k/n),\phi_1(x;k/n),\ldots \}$ form a Hilbert basis of $L^2(\R)$. Hence, every function $g \in L^2(\R)$ can be written as
\[ g(x)  = \sum_{j=0}^\infty \left( \int_{-\infty}^\infty g(y) \phi_j(y;k/n) \ dy \right)  \phi_j(x;k/n) \qquad (k=1,...,n-1) \]
in the $L^2$-sense. If we assume furthermore that $g$ is a Schwartz function, then the above series converges pointwise for all $x$. Therefore, we can apply the Dominated Convergence Theorem to obtain
\begin{multline*}
\frac{1}{\sqrt{\pi}} \int_{-\infty}^{\infty} K\left(x_{n-1},x_{n};\frac{n-1}{n}\right) e^{-\frac{a_n}{2} \, x_n^2} \, dx_{n} \\
= \sum_{j=0}^\infty  \lambda_j \left( \frac{n-1}{n} \right) \left[ \int_{-\infty}^{\infty}  \phi_j\left(x_{n}; \frac{n-1}{n} \right) e^{-\frac{a_n}{2}} \, x_n^2\, dx_{n} \right] \ \phi_j\left(x_{n-1};\frac{n-1}{n}\right) 
\end{multline*}
Iterating this with  $\frac{1}{\sqrt{\pi}} K(x_{n-2},x_{n-1};\frac{n-2}{n}),  \ldots,  \frac{1}{\sqrt{\pi}} K(x_{1},x_{2};\frac{1}{n})$, and using the almost commuting relations in \eqref{comm},  it follows  
\begin{align*}
 \frac{\sqrt{\pi}}{\sqrt{D_n(f)}} &= \sum_{j=0}^\infty \left[ 
 \prod_{k=1}^{n-1} \lambda_j\left(\frac{k}{n}\right) \cdot
\int_{-\infty}^{\infty} \phi_j \left(x;\frac{1}{n} \right) e^{-\frac{a_1}{2}x^2} dx \right.
\nonumber \\
& \qquad \times \left. \int_{-\infty}^{\infty} \phi_j \left(x;\frac{n-1}{n} \right) e^{-\frac{a_n}{2}x^2} dx \right]
 + \Oo \left( \frac{1}{n} \right).
\end{align*}
From the non-degeneracy of the eigenvalues, the above series is dominated by the leading term $(j=0)$ as $n$ gets arbitrary large.  Using the facts that $a_{n} \rightarrow f(1)$ and $a_1 \rightarrow f(0)$, we conclude 
\begin{align}
 \frac{\sqrt{\pi}}{\sqrt{D_n(f)}} &= 
 \prod_{k=1}^{n-1} \lambda_0\left(\frac{k}{n}\right) \cdot
\int_{-\infty}^{\infty} \phi_0(x;0) e^{-\frac{f(0)}{2}x^2} dx 
\nonumber \\
& \qquad \times  \int_{-\infty}^{\infty} \phi_0(x;1) e^{-\frac{f(1)}{2}x^2} dx 
 + o(1).
\label{detTna2}
\end{align}
From the expressions for $\lambda_0$ and $\phi_0$ and after evaluating the integrals in \eqref{detTna2}, we finally arrive at
\begin{align}
\lim_{n\rightarrow\infty} \frac{D_n(f)}{\prod_{k=1}^{n-1} \left( \frac{a_k}{2} + \sqrt{\left(\frac{a_k}{2}\right)^2 - 1}\right)} 
&= \frac{1}{4} \cdot \frac{f(0)+\sqrt{f(0)^2 - 4}}{ \sqrt[4]{f(0)^2 - 4}}
\nonumber
\\
& \qquad \times 
 \frac{f(1)+\sqrt{f(1)^2 - 4}}{ \sqrt[4]{f(1)^2 - 4}}
 \label{kac2}
\end{align}
Now we need to evaluate the product in the denominator.  We write it as the exponential of the sum of logarithms.
Expanding
\begin{align*}
 \log\left(\frac{a_k}{2}+\sqrt{\left(\frac{a_k}{2}\right)^2-1}\right) & = \log\left(\frac{f\left(\frac{k}{n}\right) + \sqrt{f^2\left(\frac{k}{n}\right)-4}}{2}\right)
\\
& \quad + \frac{1}{2n}\frac{f'\left(\frac{k}{n}\right)}{\sqrt{f^2\left(\frac{k}{n}\right)-4}} + \Oo\left(\frac{1}{n^{2}}\right)
\end{align*}
Now, we have the Riemann sum:
\begin{align}
 \sum_{k=1}^{n-1} \frac{1}{2n}\frac{f'\left(\frac{k}{n}\right)}{\sqrt{f^2\left(\frac{k}{n}\right)-4}} 
&= \frac{1}{2} \int_0^1 \frac{f'\left(s\right)}{\sqrt{f^2\left(s\right)-4}} ds  + \Oo\left(\frac{1}{n}\right)
\nonumber \\
&= \left. \frac{1}{2} \log\left(\frac{f(s)+\sqrt{f^2(s)-4}}{2}\right) \right|_{s=0}^{s=1} + \Oo\left(\frac{1}{n}\right)
\label{RS1}
\end{align}
Next we use the Euler-Maclaurin formula
\beq 
\sum_{k=1}^{n-1}g\left(\frac{k}{n}\right) = n\int_0^1 g(s)\, ds - \frac{g(0)+g(1)}{2} + \Oo\left(\frac{1}{n}\right) \nonumber
\label{EM}
\eeq
with $g(x) = \log\left(\frac{f(x)+\sqrt{f^2(x)-4}}{2}\right)$ to get
\begin{align}
\sum_{k=1}^{n-1} \log\left(\frac{f\left(\frac{k}{n}\right) + \sqrt{f^2\left(\frac{k}{n}\right)-4}}{2}\right)
&= n\int_0^1 \log\left(\frac{f(s)+\sqrt{f^2(s)-4}}{2}\right) ds
\nonumber
\\
& \quad -\frac{g(0)+g(1)}{2} + \Oo\left(\frac{1}{n}\right)
\label{EM2}
\end{align}
Combining \eqref{RS1} and \eqref{EM2} gives us
\begin{align*}
\sum_{k=1}^{n-1} \log\left(\frac{a_k}{2}+\sqrt{\left(\frac{a_k}{2}\right)^2-1}\right) 
& = n\int_0^1 \log\left(\frac{f(s)+\sqrt{f^2(s)-4}}{2}\right) ds
\nonumber
\\
& \quad -  \log\left(\frac{f(0)+\sqrt{f^2(0)-4}}{2}\right) + \Oo\left(\frac{1}{n}\right)
\end{align*}
Thus, the product in the denominator of \eqref{kac2} is
\[\prod_{k=1}^{n-1} \left( \frac{a_k}{2} + \sqrt{\left(\frac{a_k}{2}\right)^2 - 1}\right)
= G(f)^n\left\{\left(\frac{f(0)+\sqrt{f^2(0)-4}}{2}\right)^{-1} + o(1)\right\}
\]
Combining this with \eqref{kac2} gives us \eqref{kac1}.
\end{proof}

\begin{rem}
In the above proof, 
it is sufficient for $f\in C^{1+\alpha}([0,1])$ for some $\alpha > 0$.  For the approximation of the integral
\eqref{dfintegral} in terms of symmetric kernels, it is enough for the error in \eqref{fexp} to be $\Oo\left(n^{-1-\alpha}\right)$ for any $\alpha > 0$.  $f\in C^{1+\alpha}$ guarantees this condition.  For the remainder of the proof, one has to keep track of the error and introduce a modest generalization of the Eulr-Maclaurin formula.
\label{1palpha}
\end{rem}

For the proof of Theorem~\ref{scaletheorem}, we need the following lemma.
\begin{lem}   \label{shiftlemma}
Let $g$ be twice differentiable with a bounded second derivative on  an open interval $I$ containing $[0,1]$. Fix $\ve\in\mathbb{R}$.  Then
\begin{align}
\sum_{k=1}^{n-1} g\left(\frac{k-1+\ve}{n}\right) & = n \int_0^1 g(x) dx 
\nonumber
\\
& \quad + \left( \ve - \frac{3}{2}\right) g(1) + \left( \frac{1}{2}-\ve\right) g(0) + \Oo\left(\frac{1}{n}\right) \nonumber
\label{shiftEM}
\end{align}
\end{lem}

\begin{proof}
Let $n$ be large enough so that $\left[\frac{\ve -1}{n}, \frac{n-2+\ve}{n}\right] \subset I$.  Then, 
by the Euler-Maclaurin formula
\begin{align*}
\sum_{k=1}^{n-1} g\left(\frac{k-1+\ve}{n}\right) &= n \int_{\frac{\ve-1}{n}}^{\frac{n-2+\ve}{n}} g(x) dx
+\frac{1}{2}\left(g(1)-g(0)\right) + \Oo\left(\frac{1}{n}\right)
\\
&= n\int_0^1 g(x) dx + n \int_1^{\frac{n-2+\ve}{n}} g(x) dx + n \int_{\frac{\ve-1}{n}}^0 g(x) dx
\\
& \quad +\frac{1}{2}\left(g(1)-g(0)\right) + \Oo\left(\frac{1}{n}\right)
\\
&= n\int_0^1 g(x) dx + (\ve -2) g(1) + (1-\ve) g(0)
\\
& \quad +\frac{1}{2}\left(g(1)-g(0)\right) + \Oo\left(\frac{1}{n}\right)
\end{align*}
from which the result follows.
\end{proof}

\begin{proof}[Proof of Theorem~\ref{scaletheorem}]
The proof of Theorem~\ref{kacformula} carries through almost without change as long as we let $n$ be large enough so that $\left[\frac{\ve -1}{n}, \frac{n-2+\ve}{n}\right] \subset I$,  replace $f(k/n)$ with  $f((k-1+\ve)/n)$, and let
\[ a_k = f\left(\frac{k-1+\ve}{n}\right) + \frac{1}{2n}f'\left(\frac{k-1+\ve}{n}\right)
\]
The only difference is equation \eqref{EM2}.  There we will get $\sum_{k=1}^{n-1} g\left(\frac{k-1+\ve}{n}\right)$ for  $g(x) = \log\left(\frac{f(x)+\sqrt{f^2(x)-4}}{2}\right)$ .  Then we apply Lemma~\ref{shiftlemma} in this sum, so that
\begin{align*}
\sum_{k=1}^{n-1} \log\left(\frac{a_k}{2}+\sqrt{\left(\frac{a_k}{2}\right)^2-1}\right) 
& = n\int_0^1 \log\left(\frac{f(s)+\sqrt{f^2(s)-4}}{2}\right) ds
\nonumber
\\
& \quad + \left(\ve - 1\right) g(1) - \ve g(0) + \Oo\left(\frac{1}{2}\right)
\end{align*}
Thus, the product in the denominator of \eqref{kac2} becomes
\begin{align*}
\prod_{k=1}^{n-1}  & \left( \frac{a_k}{2} + \sqrt{\left(\frac{a_k}{2}\right)^2 - 1}\right)
= G(f)^n
\\
& \times \left\{\left(\frac{f(0)+\sqrt{f^2(0)-4}}{2}\right)^{-\ve} \left(\frac{f(1)+\sqrt{f^2(1)-4}}{2}\right)^{\ve-1} + o(1)\right\}
\end{align*}
Combining this with \eqref{kac2} gives us \eqref{dsoe}.
\end{proof}
\medskip
To prove Theorem~\ref{jumptheorem}, we need the following generalization of the Euler-Maclaurin formula.  
\begin{lem} \label{integraljump}
(i) Suppose $g$ is twice differentiable with a bounded second derivative  except for $r<\infty$  jump discontinuities at $0<c_1 < c_2 < \cdots < c_r < 1$.  Assume that both sided limits exist and are finite, and that $g$ is left-continuous at these points.
Then
\begin{align}
\sum_{k=1}^{n-1} g\left(\frac{k}{n}\right) &= n\int_0^1 g(x)\, dx - \frac{g(0)+g(1)}{2} \nonumber \\
 & \quad + \sum_{j=1}^r \left( \{nc_j\} - \frac{1}{2}  \right) \left[g(c_j+)-g(c_j-)\right] \ +   \Oo\left(\frac{1}{n}\right)   
  \label{try}
\end{align}

\noindent (ii) If $g$ is right-continuous at $c_j$, then the formula \eqref{try} holds with $\{nc_j\}$ replaced by $\{nc_j\}'$.
\end{lem}

\begin{proof}
(i)  Suppose $g$ has a single jump discontinuity at $c\in (0,1)$ and is left-continuous at $c$.
Apply the Euler-Maclaurin formula \eqref{EM}  for $C^2$ functions 
 to each of the following sums:
\begin{align}
\sum_{k=1}^{n-1}g\left(\frac{k}{n}\right) &= \sum_{k=1}^{\lfloor nc\rfloor} g\left(\frac{k}{n}\right)
+\sum_{k=\lfloor nc\rfloor +1}^{n-1} g\left(\frac{k}{n}\right)
\nonumber
\\
&= n\int_0^{\lfloor nc\rfloor/n} g(x)\, dx + \frac{1}{2} \left[ g\left( \frac{\lfloor nc\rfloor}{n}\right) - g(0)\right] 
\nonumber
\\[.05in]
& \quad +  n \int_{(\lfloor nc\rfloor+1) /n}^1 g(x) \,dx + \frac{1}{2} \left[ -f(1) +  g\left(\frac{\lfloor nc\rfloor+1}{n}\right) \right] + \Oo\left(\frac{1}{n}\right)
\nonumber
\\[.1in]
&= n \int_{0}^1 g(x) \,dx + \frac{1}{2} \left[ -g(1)-g(0) + g\left(\frac{\lfloor nc\rfloor}{n}\right) +  g\left(\frac{\lfloor nc\rfloor+1}{n}\right)  \right] 
\nonumber
\\[.05in]
& \quad -   n \int_{\lfloor nc\rfloor/n}^{(\lfloor nc\rfloor+1) /n} g(x) \,dx  + \Oo\left(\frac{1}{n}\right)
\label{intapprox}
\end{align}
Now,
\begin{align*}
\int_{\lfloor nc\rfloor/n}^{(\lfloor nc\rfloor+1) /n} g(x) \,dx
 &= \left( c - \frac{\lfloor nc\rfloor }{n} \right) g(c-) 
+ \left( \frac{\lfloor nc \rfloor + 1}{n} - c\right) g(c+) 
+ \Oo\left(\frac{1}{n^2}\right)
\\
&= \frac{1}{n}\left\{ \{nc\} \,  \left[g(c-) - g(c+)\right] + g(c+) + \Oo\left(\frac{1}{n}\right)\right\}
\end{align*}
Combining this with \eqref{intapprox}, and using
\[  g\left(\frac{\lfloor nc\rfloor }{n}\right) = g(c-) + \Oo\left(\frac{1}{n}\right), \quad \mbox{ and } \quad
 g\left(\frac{\lfloor nc\rfloor +1 }{n}\right) = g(c+) + \Oo\left(\frac{1}{n}\right),
 \] 
 establishes the result for a single jump discontinuity.

\medskip 
When there are $r>1$ jump discontinuities, we break the sum into $r+1$ parts:
 \[ \sum_{k=1}^{n-1}g\left(\frac{k}{n}\right) = \left(\sum_{k=1}^{\lfloor nc_1\rfloor }
+\sum_{k=\lfloor nc_1\rfloor +1}^{\lfloor nc_2\rfloor}
+ \cdots
+ \sum_{k=\lfloor nc_r\rfloor +1}^{n -1}\right) g\left(\frac{k}{n}\right)
\]
We then apply the Euler-Maclaurin formula to each of the sums.  The proof for each of the sums is identical to that when there is a single jump discontinuity.

\smallskip
\noindent
(ii)  If $f$ is right-continuous at $c$, then we have to break the sum into
\[ \sum_{k=1}^{n-1}g\left(\frac{k}{n}\right) = \sum_{k=1}^{\lceil nc\rceil -1 } g\left(\frac{k}{n}\right)
+\sum_{k=\lceil nc\rceil}^{n-1} g\left(\frac{k}{n}\right)
\]
The rest of the proof proceeds as in  case (i), \textit{mutatis mutandis}.
\end{proof}

With the above lemma, we can now proceed to the proof of Theorem~\ref{jumptheorem} by making the necessary changes to the proof of Theorem~\ref{kacformula}.
\begin{proof}[Proof of Theorem~\ref{jumptheorem}]
(i) Suppose $f$ has a single jump discontinuity at $c\in (0,1)$, and is left-continuous at $c$.  Extend the restriction  of $f$ on $(c,1]$ to a $C^2$ function 
 $\tilde{f}(s)$ on $[0,1]$, and  let $m$ be the index such that
\[ \frac{m}{n} \leq c < \frac{m+1}{n}
\]
 Then $f(k/n)$ are as in \eqref{fexp} for all $k$ except $k=m+1$, where we have
\begin{align*}
f\left(\frac{m+1}{n}\right) &= 
\frac{1}{2}\tilde{f}\left(\frac{m}{n}\right) + \frac{1}{4n}\tilde{f}'\left(\frac{m}{n}\right)
+ \frac{1}{2}f\left(\frac{m+1}{n}\right)
+ \frac{1}{4n}f'\left(\frac{m+1}{n}\right)
\\
&\quad + \Oo\left(\frac{1}{n^{2}}\right)
\\
&= \frac{1}{2}\tilde{a}_m + \frac{1}{2}a_{m+1} + \Oo\left(\frac{1}{n^{2}}\right)
\end{align*}
where 
\[ \tilde{a}_m = \tilde{f}\left(\frac{m}{n}\right) + \frac{1}{2n}\tilde{f}'\left(\frac{m}{n}\right)
\]
Then the sum \eqref{fsum} has to be modified by adding to it the term
\[ \frac{1}{2} \left( \tilde{a}_m - a_m\right)x_{m+1}^2
\]
(Note that $\tilde{a}_m \rightarrow f(c+)$ and $a_m \rightarrow f(c-)$ as $n\rightarrow\infty$.)
This modifies \eqref{detTna} to
\begin{align}
\frac{\sqrt{\pi}}{\sqrt{D_n(f)}}
 &=  \int_{-\infty}^{\infty} \cdots \int_{-\infty}^{\infty} e^{-\frac{a_1}{2}x_1^2} \prod_{k=1}^{n-1}K\left(x_k, x_{k+1};\frac{k}{n}\right) 
 \nonumber
\\
& \quad \times e^{-\frac{1}{2}(\tilde{a}_m - a_m)x_{m+1}^2}  e^{-\frac{a_n}{2} x_n^2}\,  dx_1\cdots dx_n 
 + o(1) \nonumber
\end{align}

Expanding again in eigenfunctions, eqn \eqref{detTna2} holds as long as we multiply the RHS   by
\begin{align} 
& \int_{-\infty}^{\infty} e^{-\frac{1}{2}(\tilde{a}_m - a_m)x^2} \phi_0\left(x; \frac{m}{n}\right) \phi_0\left(x; \frac{m+1}{n}\right) dx
\nonumber
\\
&\qquad =  \frac{\sqrt{2}  \left[ (f(c-)^2-4)(f(c+)^2-4)\right]^{1/8}}{\sqrt{f(c-)-f(c+) + \sqrt{f(c-)^2-4} + \sqrt{f(c+)^2-4}}} + o(1) 
\label{efunmod}
\end{align}
Thus eqn \eqref{kac2} holds as long as we multiply the RHS by the $-2$ power of the above expression.

Next we apply Lemma~\ref{integraljump} in the calculation of
the product in the denominator of the LHS of \eqref{kac2}.  With
$ g(x) = \log\left(\frac{f(x)+\sqrt{f^2(x)-4}}{2}\right)$, we have
\begin{align*}
\sum_{k=1}^{n-1} \log &\left(\frac{a_k}{2} + \sqrt{\left(\frac{a_k}{2}\right)^2-1}\right)
\\
&= \sum_{k=1}^{n-1}g\left(\frac{k}{n}\right) + \frac{1}{2}[g(1)-g(0)] 
 - \frac{1}{2}[g(c+)-g(c-)] + o(1)
\\
&= n\int_0^1 g(x) dx - \frac{1}{2}[g(1)+g(0)] 
 + \left(\{nc\}-\frac{1}{2}\right)[g(c+)-g(c-)]
\\
& \quad + \frac{1}{2}[g(1)-g(0)] -\frac{1}{2}[g(c+)-g(c-)] + o(1)
\\
&= n\int_0^1 g(x) dx - g(0) 
 + \left(\{nc\}-1\right)[g(c+)-g(c-)] + o(1)
\end{align*}
Taking exponentials, we obtain
\begin{align}
& \prod_{k=1}^{n-1}\left(\frac{a_k}{2} + \sqrt{\left(\frac{a_k}{2}\right)^2-1}\right) = G(f)^n
\nonumber \\
 & \
\times   \left\{ \left(\frac{f(0)+\sqrt{f^2(0)-4}}{2}\right)^{-1}  \left[ \frac{f(c+)+\sqrt{f(c+)^2-4}}{f(c-)+\sqrt{f(c-)^2-4}}\right]^{\{nc\}-1}  +  o(1)\right\}
\label{prod3}
\end{align}
Now, taking \eqref{kac2}, multiplying the RHS by the $-2$ power of \eqref{efunmod}, and combining this with \eqref{prod3}
 gives us \eqref{jumpeqn} when there is a single left-continuous  jump discontinuity.  When there are $r$ such discontinuities, we simply apply the same reasoning to each of them.  

(ii) The case when $f$ is right-continuous is similar.  We only have to change the index $m$ to be such that 
\[ \frac{m}{n} <c \leq \frac{m+1}{n}
\]
The calculation of the extra term on the RHS of \eqref{kac2} is then identical.  The calculation of the denominator on the LHS of \eqref{kac2} proceeds in the same way, by using the other part of Lemma~\ref{integraljump}.
\end{proof}


\end{document}